\documentclass[12pt,a4paper]{amsart}
\usepackage[margin=2cm]{geometry}
\usepackage{amsmath,amsfonts,amssymb,amsthm}
\usepackage[utf8]{inputenc}
\usepackage{graphicx}
\usepackage{xcolor}
\usepackage{pinlabel}
\usepackage{fourier}\usepackage{microtype}
\usepackage{caption}\usepackage{verbatim}
\usepackage{tikz,everypage}
\numberwithin{equation}{section}

\newcommand{\Q}{\mathbb{Q}}

\newcommand{\CP}{\mathbb{CP}}

\newcommand{\Z}{\mathbb Z}

\newcommand{\la}{\lambda}

\newcommand{\x}{\times}

\newcommand{\del}{\partial}

\newcommand{\co}{\thinspace\colon}
\newcommand{\lra}{\longrightarrow}
\newcommand*\sm[1]{\left(\begin{smallmatrix}#1\end{smallmatrix}\right)}
\newcommand*\mat[1]{\begin{pmatrix}#1\end{pmatrix}}

\newtheorem{thm}{Theorem}
\newtheorem{lemma}[thm]{Lemma}

\newtheorem{rem}[thm]{Remark}
\newtheorem{prop}[thm]{Proposition}

\theoremstyle{definition}

\newtheorem{rmk}{Remark}

\newtheorem{quest}{Question}
\newtheorem*{ack}{Acknowledgments}

\title{On Stein rational balls smoothly but not symplectically embedded in $\CP^2$}

\begin{document}
\author{Paolo Lisca}
\author{Andrea Parma}
\address{Dipartimento di Matematica, Largo Bruno Pontecorvo 5, 56127 Pisa, Italy} 
\email{paolo.lisca@unipi.it,\ andrea.parma94@gmail.com}	
%\subjclass[2000]{57R40 (Primary) 57K43, 57R17 (Secondary)}
\date{\today}

\begin{abstract}
We extend recent work of Brendan Owens by constructing a doubly infinite family of Stein rational homology balls which can be smoothly but not symplectically embedded in $\CP^2$. 
\end{abstract}

\maketitle	

\section{Introduction and statement of results}\label{s:intro} 

Let $p>q\geq 1$ be coprime integers and $B_{p,q}$ the rational homology ball smoothing of the quotient singularity $\frac{1}{p^2}(pq-1,1)$. Using results by Khodorovskiy~\cite{Kh13} it is not hard to show~\cite[\S~2.1]{ES18} that if the positive  integers $p_1$, $p_2$ and $p_3$ form a {\em Markov triple}, that is $p_1^2+p_2^2+p_3^2 = 3 p_1 p_2 p_3$, 
then there are pairwise disjoint symplectic embeddings
\begin{equation}\label{e:disjemb} 
B_{p_i,q_i}\subset\CP^2,\quad i=1,2,3,
\end{equation}
where $q_i=\pm 3p_j/p_k \bmod p_i$ with $\{i,j,k\}=\{1,2,3\}$. Note that the sign is irrelevant because $B_{p,q}$ is symplectomorphic to $B_{p,p-q}$~\cite[Remark~2.8]{ES18}. The existence of the simultaneous symplectic embeddings~\eqref{e:disjemb} comes from the fact that when $(p_1,p_2,p_3)$ is a Markov triple there is a $\Q$-Gorenstein smoothing to $\CP^2$ of the weighted projective space $\CP(p_1^2,p_2^2,p_3^2)$. It is not possible to construct more than three disjoint symplectic embeddings using smoothings of singular surfaces. In fact, Hacking and Prokhorov~\cite{HP10} showed if $X$ is a projective surface with quotient singularities which has a $\CP^2$ smoothing, then $X$ is a $\Q$-Gorenstein deformation of such a weighted projective plane. Evans and Smith~\cite[Theorem~1.2]{ES18} generalized this result to the symplectic category, showing that if $B_{p_i,q_i}\subset\CP^2$, $i=1,...,N$ is a collection of pairwise disjoint symplectic embeddings then $N\leq 3$, the $p_i$ belong to Markov triples and the $q_i$'s must satisfy certain constraints. In particular, if $B_{p,q}\subset\CP^2$ is a symplectic embedding then $p$ must belong to a Markov triple and divide $q^2+9$.

Owens~\cite[Theorem~1]{Ow20} recently proved the existence of {\em smooth} embeddings 
\[
B_{F_{2n+1}, F_{2n-1}}\subset\CP^2
\]
for each $n\geq 1$, where $F_{2n-1}$ denotes the odd Fibonacci number, recursively defined by 
\[
F_1 = 1,\ F_3 = 2,\ F_{2n+3} = 3F_{2n+1} -  F_{2n-1}.
\]
Moreover, he showed that the pair $(F_{2n+1},F_{2n-1})$ satisfies the Evans-Smith constraints only if $n=1$, and therefore that $B_{F_{2n+1}, F_{2n-1}}$ does not embed {\em symplectically} in $\CP^2$ for $n>1$.

In this paper we extend Owens' family of smooth embeddings to a two-parameter family of smooth embeddings $B_{p,q}\subset\CP^2$ such that $B_{p,q}$ cannot be symplectically embedded in $\CP^2$. 

Recall that to a string of integers $s=(a_1,\ldots, a_n)$ is uniquely associated a smooth, oriented 4-dimensional plumbing $P(s) = P(a_1,\ldots, a_n)$. When $a_i\geq 2$ for each $i$, the Hirzebruch-Jung continued fraction 
\[
[s] = [a_1,\ldots, a_n] = a_1 - \cfrac{1}{a_2-\cfrac{1}{\cdots - \cfrac{1}{a_n}}}
\]
is well-defined, and the oriented boundary of $P(s)$ is the lens space $L(p,p-q)$, where $\frac{p}{q} = [a_1,\ldots, a_n]$. 

Given integers $k\geq -1$ and $m\geq 1$, define
\[
s_{k,m} := (2,(2^{[m-1]},m+2)^{[k+1]},2,2,(2^{[m-1]},m+2)^{[k+1]}),
\]
where $x^{[n]}$ means $x$ repeated $n$ times if $n>0$ and omitted when $n=0$. 
We observe in Remark~\ref{r:type-of-lens-space} below that the lens space $L(s_{k,m})=\del P(s_{k,m})$ is of the form $L(p^2,pq-1)$ for some $p>q\geq 1$. We denote by $B(s_{k,m})$ the corresponding rational homology ball $B_{p,q}$. 

When $m=1$, $s_{k,1} = (2,3^{[k+1]},2,2,3^{[k+1]})$ and using Riemenschneider's point rule~\cite{Ri74} one can check that if $\frac{p}{q} = [s_{k,1}]$ then 
$\frac{p}{p-q} = [3^{[k+1]},5,3^{[k]},2]$. Moreover, the proof of~\cite[Theorem~1]{Ow20} shows that 
\[
\frac{F_{2k+5}^2}{F_{2k+5} F_{2k+3}-1} = [3^{[k+1]},5,3^{[k]},2], 
\]
therefore $B(s_{k,1}) = B_{F_{2k+5}, F_{2k+3}}$. Therefore Owens' family is precisely the one-parameter subfamily $\{B(s_{k,1})\}_{k\geq -1}$. Notice that the string $s_{-1,m}$ reduces to $(2,2,2)$ for each $m$. In this case the ball  $B(s_{-1,m})=B_{2,1}$ embeds symplectically in $\CP^2$ as the complement of a neighborhood of a smooth conic. The following is our main result.

\begin{thm}\label{t:main1}
	Let $k\geq -1$ and $m\geq 1$, $m$ odd. Then, 
	\begin{enumerate}
		\item[(1)]
		$B(s_{k,m})$ smoothly embeds in $\CP^2$;
		\item[(2)]
		$B(s_{k,m})$ does not symplectically embed in $\CP^2$ if $k\geq 0$.
	\end{enumerate}
\end{thm}

\begin{rmk}\label{r:m=1}
	Theorem~\ref{t:main1} is equivalent to~\cite[Theorem~1]{Ow20} when $m=1$.
\end{rmk} 

We prove Theorem~\ref{t:main1}(1) by showing that, for each $k\geq -1$ and $m\geq 1$, there is a smooth decomposition  
\[
\CP^2 = S^1\x D^3\cup h_1\cup h_2\cup h_3\cup S^1\x D^3,
\]
where $h_i$, for $i=1,2,3$ is a $2$-handle and $B(s_{k,m}) = S^1\x D^3\cup h_2$. Theorem~\ref{t:main1}(2) follows from~\cite[Theorem~1]{Ow20} if $m=1$, while for $m>1$ 
we show that $\del B(s_{k,m})$ is of the form $L(p^2,pq-1)$, where $p$ does not divide $q^2+9$. The conclusion follows by the results of~\cite{ES18}. 

In~\cite{Ow20} Owens also proves another result (Theorem~2), which 
states that {\em a disjoint union of two or more of the balls $B(s_{k,1})$ cannot be 
	smoothly embedded in $\CP^2$}. This is viewed in~\cite{Ow20} as mild support to a 
conjecture of Koll\'ar~\cite{Ko08}, which would imply that at most three of the rational 
balls $B(s_{k,1})$ may embed smoothly and disjointly in $\CP^2$. It is therefore 
natural to ask whether the analogue of~\cite[Theorem~2]{Ow20} holds for our extended family or rational balls:
\begin{quest}\label{q:1}
	Can a disjoint union of two or more balls $B(s_{k,m})$ be smoothly embedded in $\CP^2$ ? 
\end{quest}
%The construction used in the proof of Theorem~1 does not immediately yield a positive answer to Question~\ref{q:1}. In fact, we have checked that no disjoint union $S^1\x D^3\cup h_i\cup h^*_j\cup S^1\x D^3$, where $i\neq j$ and $h^*_j$ is the $2$-handle dual to $h_j$, is of the form $B(s_{k,m})\cup B(s_{k',m'})$. 
We plan to address Question~\ref{q:1} in a future paper. This paper is organized as follows. In Section~\ref{s:prelim} we fix notation and collect some preliminary material. Section~\ref{s:theorem1} contains the proof of Theorem~\ref{t:main1}. 

\begin{ack}
	The authors wish to thank Brendan Owens for a stimulating email exchange and an anonymous referee for helpful comments. The present work is part of the MIUR-PRIN project 2017JZ2SW5. 
\end{ack}

\section{$SL_2(\Z)$-framed chain links and $SL_2(\Z)$-slam-dunks}\label{s:prelim}

Given a string of integers $s=(a_1,...,a_n)$, let 
\[
K = K_1\sqcup\ldots\sqcup K_n\subset S^3
\]
be a chain link consisting of $n$ oriented, framed unknots, with framing coefficients specified by $s$. Performing Dehn surgery along each $K_i$ with coefficient $a_i$ gives 
rise, in the notation of Section~\ref{s:intro}, to the lens space $L(s)=\del P(s)$. 
%\subsection{$SL_2(\Z)$-framed links and $SL_2(\Z)$-slum-dunks}\label{ss:sl2}
We shall need to keep track of detailed information about the gluing maps involved in the Dehn surgeries on the components of $K$. In order to do that we are going to view the framed link $K$ as an {\em $SL_2(\Z)$-framed link} in the sense of~\cite[Appendix]{KM94}, although we will use our own notation rather than the notation from~\cite{KM94}.  

Let $Y := S^3\setminus N$ be the complement of a tubular neighborhood $N := N_1 \sqcup\ldots\sqcup N_n$ of $K_1 \sqcup\dots\sqcup K_n$. We can express $L(s)$ as the result of gluing $n$ solid tori $V_1,...,V_n$ to $Y$. The gluing maps $\varphi_i: \partial N_i \rightarrow \partial V_i$ are determined up to isotopy by $2 \times 2$ matrices if we specify, for each of the tori $\del N_i$ and $\del V_i$, two oriented curves that generate its first homology group -- we identify such oriented curves with their homology classes and the maps $\varphi_i$ with the induced maps in homology. We can do this as follows:
\begin{itemize}
	\item orient $K_1,..., K_n$ so that $\text{lk}(K_i,K_{i+1})=-1 \; \forall i$;
	\item in each $\partial N_i$, choose a canonical longitude $\lambda_i$ with the same orientation as $K_i$, and an oriented meridian $\mu_i$ that winds around $K_i$ according to the right-hand convention;
	\item regarding each $V_i$ as the tubular neighborhood of an unknot in $S^3$, choose a canonical longitude $\ell_i$ and a meridian $m_i$ in $\partial V_i$ as above;
	\item for each $i$, choose the basis $(\lambda_i,\mu_i)$ for $H_1(\partial N_i)$ and the basis 
	$(\ell_i, m_i)$ for $H_1(\partial V_i)$.
\end{itemize}
Notice that, with these assumptions, $Y$ and $V_1,...,V_n$ have compatible orientations if and only if the matrices representing $\varphi_1,...,\varphi_n$ with respect to the bases $(\lambda_i,\mu_i)$ and $(\ell_i, m_i)$ have determinant $1$. With this in mind, and recalling that each $m_i$ must be sent by $\varphi_i^{-1}$ to $\lambda_i+a_i \mu_i$, we can choose
$\varphi_i$ with matrix 
\begin{equation}\label{e:2x2}
A_{a_i}, \quad \text{where }A_m \text{ denotes the matrix }\begin{pmatrix}
m & -1\\
1 & 0
\end{pmatrix}\in SL_2(\Z)\ \text{and}\ m \in\Z.
\end{equation}
After these choices, each component $K_i$ is decorated with the matrix $A_{a_i}$ rather than 
simply with the integer $a_i$, and $K$ becomes an $SL_2(\Z)$-framed link. Moreover, a 
presentation $\{(K_i,A_{a_i})\}_{i=1}^n$ can be modified via {\em $SL_2(\Z)$-slam-dunks} 
(cf.~\cite[Lemma~(A.2)]{KM94}). We describe these modifications using our notation in the 
following proposition. 

\begin{prop}\label{p:slamdunk} 
	Let $s=(a_1,\ldots, a_n)$ with $n>1$ and $L=L(s)$. Then, 
	\begin{enumerate}
		\item[(1)]
		For $t=1,\ldots, n$, the oriented meridian $\mu_t$ is isotopic to a curve lying in a regular neighborhood of $\del V_1\subset L$. Its homology class has coordinates, with respect to the basis $\ell_1$, $m_1$, given by the second column of $A_{a_1}\cdots A_{a_t}$.
		\item[(2)]
		for each $t = 2,\ldots,n$ the $SL_2(\Z)$-framed link presentation $\{(K_i,A_{a_i})\}_{i=1}^n$ of $L$ can be modified into another presentation of $L$ given by $\{(K_i,B_i)\}_{i=t}^n$, where 
		\[
		B_t = A_{a_1}\cdots A_{a_t}\quad\text{and}\quad B_i = A_{a_i}\quad\text{for $i>t$}.
		\]
		\item[(3)]
		$L$ is orientation-preserving diffeomorphic to $L(p,p-q)$, where $\sm{p\\q}$ is the first column of $A_{a_1}\cdots A_{a_n}$.
	\end{enumerate}
\end{prop}

\begin{proof} 
	We first describe the case $t=2$. 
	Let $L'$ be the lens space arising from Dehn surgeries along all the components of the chain link except $K_1$, so that $L:=L(s)$ is obtained from $L'$ by doing the remaining surgery along $K_1$. Since $K_1$ is a meridian of $K_2$, we can isotope it, as an oriented knot in $L'$, to $-\mu_2=\varphi_2^{-1}(\ell_2)$ and then to the oriented core $K'_1$ of $V_2$. See Figure~\ref{fig:isotopy}, where the blue and the red oriented curves on $\del N_2$ are mapped by $\varphi_2$, respectively, to $\ell_2$ and $m_2$. 
	\begin{figure}[ht]
		\labellist
		\hair 2pt
		\pinlabel $K_1$ at 20 530
		\pinlabel $\del N_2$ at 200 750
		\pinlabel $K_2$ at 260 660
		\pinlabel \textcolor{red}{$\la_2+a_2\mu_2$} at 460 610
		\pinlabel \textcolor{blue}{$-\mu_2$} at 370 400
		\endlabellist
		\centering
		\includegraphics[scale=0.28]{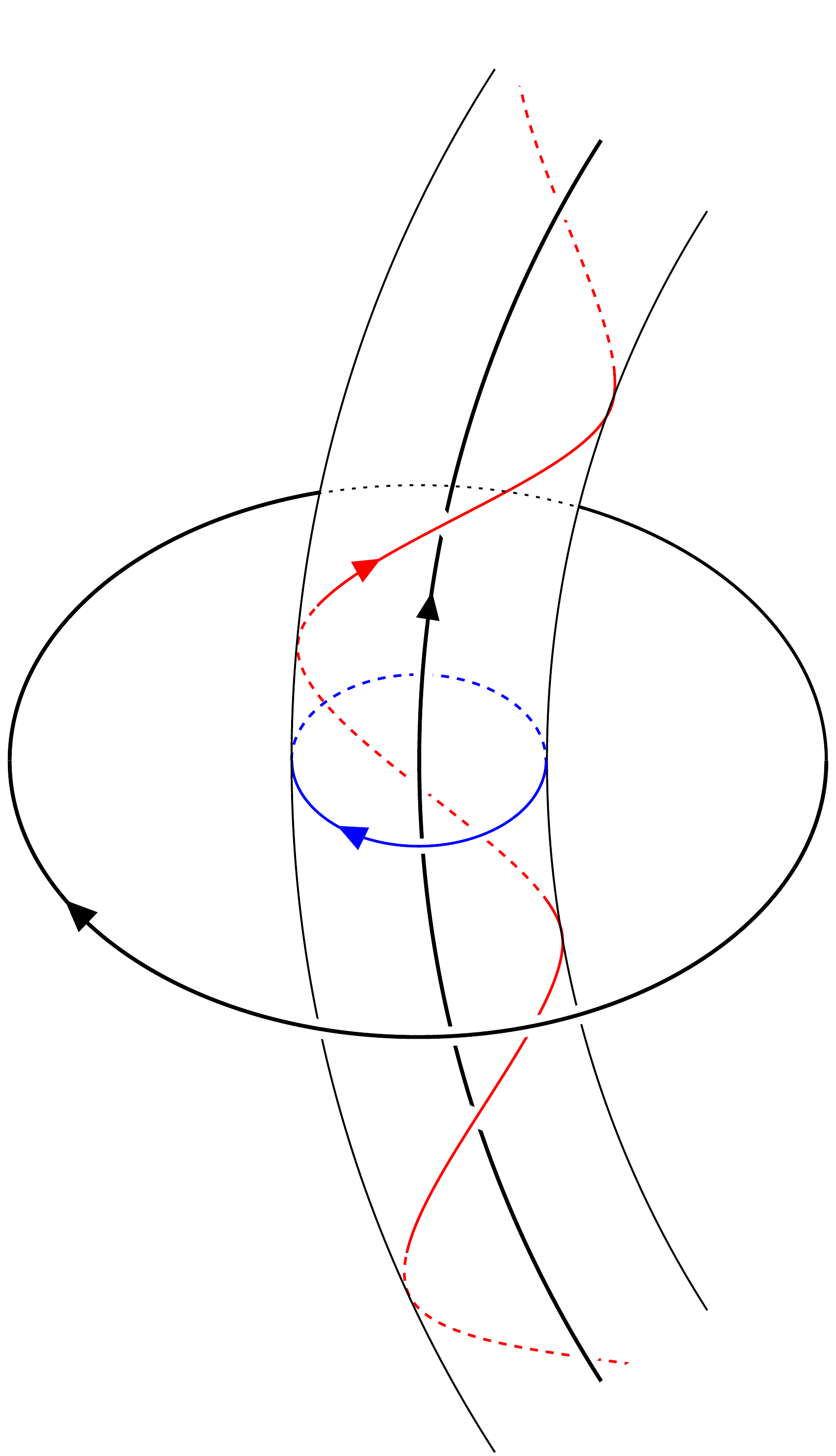}
		\caption{}%$K_1$ as a knot in $L'$. }
		\label{fig:isotopy}
	\end{figure}
	The isotopy from $K_1$ to $K'_1$ can be extended to an isotopy of tubular neighborhoods from $N_1$ to $N_1' \subset V_2$, whose boundary is parallel to $\partial V_2$. 
	Now $L$ is obtained by cutting $N_1'$ out of $V_2$ and pasting $V_1$ in its place, with the identification between $\del N'_1$ and $\del V_1$ given by a new gluing map $\varphi_1'$. Notice that, since $V_2 \setminus N'_1$ is diffeomorphic to $T^2 \times [0,1]$, we may unambiguously take $(\ell_2,m_2)$ as a basis for the domain of $\varphi'_1$ (regarded as a map between homology groups). With this assumption, $\varphi'_1$ is represented by the same matrix as $\varphi_1$. In fact, as already observed, $K'_1$ and $\ell_2$ are isotopic as oriented knots, and $K'_1$ admits $m_2$ as a right-hand-oriented meridian. Hence, $\ell_2$ and $m_2$ also play the role of the original $\lambda_1$ and $\mu_1$. Moreover, it makes sense to consider the composition 
	\[
	\varphi_1' \circ \varphi_2\co H_1(\del N_2)\stackrel{\varphi_2}{\lra} H_1(\del V_2) = 
	H_1(N'_1)\stackrel{\varphi'_1}{\lra} H_1(\del V_1), 
	\]
	which is represented by the matrix $A_{a_1} \cdot A_{a_2}$. This concludes the description of the $SL_2(\Z)$-slam-dunk when $t=2$. 
	
	We now describe the construction for $t>2$ (assuming $n \ge 3$). We first apply an $SL_2(\Z)$-slam-dunk to the first component, so that $K_1$ is removed from the chain link. Now $K_2$ is a meridian of $K_3$, so we can apply another $SL_2(\Z)$-slam-dunk along $K_2$, and so on. In general, for each $1 \le i < t$ we remove a tubular neighborhood $N'_i$ of the core of $V_{i+1}$, and identify its boundary with $\partial V_i$ via a gluing map $\varphi'_i$ represented by $A_{a_i}$ with respect to the bases $(\ell_{i+1}, m_{i+1})$ and $(\ell_i, m_i)$.
	By construction, for each $t=2,\ldots, n$ the composition of gluing maps 
	\[
	\varphi'_1 \circ\cdots \circ \varphi'_{t-1} \circ \varphi_t \co 
	H_1(\del N_t) \lra H_1(\del V_1)
	\]
	identifies $\mu_t=-\varphi_t^{-1}(\ell_t)$ with a curve whose coordinates with respect to $\ell_1$ and $m_1$ are given by the second column of $A_{a_1}\cdots A_{a_{t}}$.
	%, which equals the first column of $A_{a_1}\cdots A_{a_{t-1}}$. 
	Similarly, after gluing, the coordinates of 
	$\lambda_t$ with respect to $\ell_1$ and $m_1$ are given by the first column of 
	$A_{a_1}\cdots A_{a_{t}}$. This proves (1) and (2). 
	
	To prove (3) we choose $t=n$, so that the modified link has a single component. The result of gluing together all the ``layers'' $V_{i+1} \setminus N'_i$ for $i<n$ is diffeomorphic to $T^2 \times [0,1]$ and the boundaries of the glued-up pieces are parallel tori. Moreover, the diffeomorphism with $T^2 \times [0,1]$ can be chosen so that:
	\begin{itemize}
		\item 
		$T^2 \times \{0\}$ and $T^2 \times \{1\}$ are identified with $\del V_1$ and $\del V_n$ respectively;
		\item 
		the other parallel tori are identified with $T^2 \times \{ h\}$ for $n-2$ pairwise distinct values of $h \in (0,1)$.
	\end{itemize}
	This shows that $L(s)$ results from gluing two solid tori to $T^2 \x [0,1]$. Moreover, the boundaries of the meridian disks of the solid tori are $m_1 \subset T^2 \x \{0\}$ and $\varphi_n(\lambda_n) \subset T^2 \x \{1\}$. By construction, the curve $\varphi_n(\lambda_n)$ is isotopic to $p\ell_1+q m_1$, where 
	\[
	A_{a_1}\cdots A_{a_n} = \mat{p&a\\q&b}\in SL_2(\Z). 
	\]
	Since $\sm{p&a\\q&b}^{-1} = \sm{b&-a\\-q&p}$, $L(s)$ is the result of a Dehn surgery with framing $\frac{p}{a}$ along an unknot, where $a(p-q)\equiv 1\bmod p$. Part (3) follows immediately from the fact that the lens spaces $L(p,q)$ and $L(p,q')$ are orientation-preserving diffeomorphic when $qq'\equiv 1\bmod p$. 
\end{proof}

\section{Proof of Theorem~\ref{t:main1}}\label{s:theorem1}

The first part of Theorem~\ref{t:main1} states that $B(s_{k,m})$ smoothly embeds in $\CP^2$ if $m$ is odd. We already observed in Section~\ref{s:intro} that this is true if $k=-1$, therefore in the following we assume $k \ge 0$. 

Consider the string $s_{k,m}$ of Section 1, with $k \ge 0$ and $m$ odd, and define:
\begin{itemize}
	\item $s'_{k,m}:=(2,(2^{[m-1]},m+2)^{[k+1]},1,2,(2^{[m-1]},m+2)^{[k+1]})$;
	\item $s''_{k,m}:=(2^{[m-1]},1,m+2,(2^{[m-1]},m+2)^{[k]},2^{[m]},1,m+2,(2^{[m-1]},m+2)^{[k]})$.
\end{itemize}
It is straightforward to check that the strings $s'_{k,m}$ and $s''_{k,m}$ are both obtained from $s_{k,m}$ by changing some terms from $2$ to $1$, and that they both ``blow-down'' to $(0)$ in the sense of~\cite[Definition~2.1]{Li08}, therefore $L(s'_{k,m}) = L(s''_{k,m}) = S^1\x S^2$. 
\begin{rmk}\label{r:type-of-lens-space}
	Applying~\cite[Lemma~2.4]{Li08} to the string $s'_{k,m}$ immediately implies that $L(s_{k,m})$ is of the form $L(p^2,pq-1)$ for some $p>q\geq 1$. 
\end{rmk}
Denote by $\nu_2\subset S^1\x S^2$ the curve corresponding to the meridian of the $\langle 1\rangle$-framed unknot of the diagram associated with $s'_{k,m}$. In Section~\ref{s:prelim} the same meridian was denoted $\mu_{(k+1)m+2}$. Denote by $W$ the smooth 4-manifold with boundary obtained by viewing $S^1\x S^2$ as the boundary of $S^1\x D^3$ and attaching a 4-dimensional 2-handle along $\nu_2$ with framing $-1$. In view of~\cite[Theorem~1.1]{Li08} and~\cite[Theorem~8.5.1]{NP10}, $W$ is orientation-preserving diffeomorphic to $B(s_{k,m})$.   

We are going to prove Part (1) of Theorem~\ref{t:main1} by showing that $\CP^2$ is obtained by attaching some 
4-dimensional handles to $B(s_{k,m})$. First we attach two extra 2-handles along the 
meridians $\mu_m$ and $\mu_{(k+2)m+2}$, both with framing $1$. Notice that the indices $m$ and 
$(k+2)m+2$ give the positions where $s_{k,m}$ and $s''_{k,m}$ are different. As before, we rename 
these two meridians as $\nu_3$ and $\nu_1$ respectively, so that we encounter $\nu_1$, $\nu_2$ and 
$\nu_3$ in this order as we move along the diagram from right to left. Denote by $X$ the smooth  $4$-manifold with boundary constructed so far. If we view $\nu_1$, $\nu_2$ and $\nu_3$ as part of a surgery presentation and blow them down we get a chain of unknots whose framing coefficients are exactly given by $s''_{k,m}$. This shows that the boundary of $X$ is $S^1 \x S^2$. 
We can now add a $3-$handle and a $4-$handle to $X$ and obtain a closed $4-$manifold 
$\widehat X$. 

Our plan is to show that $\widehat X$ is diffeomorphic to $\CP^2$. 
In order to do that, we view $\nu_1, \nu_2, \nu_3\subset S^1\x S^2$ as knots 
sitting inside a regular neighborhood $U$ of $\del V_1\subset S^1\x S^2$ as in Part (1) of 
Proposition~\ref{p:slamdunk}. The proof of Proposition~\ref{p:slamdunk} shows that $U$ can be identified with $T^2\x [0,1]$ in such a way that each $\nu_i$ is identified with a 
simple closed curve $T^2 \x \{h_i\}$, where  $1>h_1>h_2>h_3>0$. Moreover, the framing induced by 
$\del N_i$ on $\nu_i$ coincides with the framing induced by $T^2\x \{h_i\}$. 
We introduce the notation
\begin{equation}\label{e:handles}
(\nu_1,\nu_2,\nu_3)=\left(\begin{pmatrix}
p_1 \\
q_1
\end{pmatrix}_{\delta_1},
\begin{pmatrix}
p_2 \\
q_2
\end{pmatrix}_{\delta_2},
\begin{pmatrix}
p_3 \\
q_3
\end{pmatrix}_{\delta_3}\right)
\end{equation}
to indicate that $\nu_i$ is $\delta_i$-framed (with $\delta_i=\pm 1$) with respect to the framing induced by $T^2\x\{h_i\}$ and the coordinates of the homology class of $\nu_i$ with respect to the basis $\ell_1$, $m_1$ are $(p_i,q_i)$. 

If we view $S^1\x S^2$ as $L(s'_{k,m})$, applying Part (2) of Proposition~\ref{p:slamdunk} for $t=n$ gives the standard presentation of $S^1\x S^2$ as $L((0))$, ie as $0$-surgery on an unknot. 
This way, $\del V_1$ gets identified with the boundary of a neighborhood of such unknot, $m_1$ with a longitude and $\ell_1$ with a meridian. 

Recall that, given a closed, oriented 3-manifold $M$ represented by a framed link with integer 
coefficients $\mathcal{L}$, there is a convenient way to represent handlebody decompositions of any 
$4$-dimensional cobordism $X$ obtained by attaching 4-dimensional 
handles to $M \times [0,1]$ along $M \times \{1\}$. In fact, the attaching curves of the 2-handles 
can always be isotoped into the complement of the glued-in solid tori of $M \x \{1\}$, so that each 
2-handle can be represented as an additional framed knot in $S^3\setminus\mathcal{L}$. The union of 
all such framed knots with $\mathcal{L}$ is a~\emph{relative Kirby diagram} representing $X$. 
This representation requires a notation which distinguishes the role played by each component. If the
framing coefficient of a knot $K$ is $n$, we are going to write it as $\langle n \rangle$ if $K$ is part of $\mathcal{L}$, and simply as $n$ if $K$ represents a 2-handle of $X$. Of course, we can 
also attach 3- and 4-handles as usual. There is a calculus for these handlebody presentations, 
usually called {\em relative Kirby calculus}. We refer the reader to~\cite[\S~5.5]{GS94} for further 
details. 

We are going to apply relative Kirby calculus to the handle decomposition of 
$\widehat X$ we just described. It turns out that the effect of sliding the handle 
$h_{\nu_i}$ attached along $\nu_i$ over (an appropriate number of copies of) 
the handle $h_{\nu_{i+1}}$ attached over $\nu_{i+1}$, for $i=1,2$, was described in~\cite[Lemma~5.1]{TY12}. 
In terms of our Notation~\eqref{e:handles}, the action of such 
handle slides on the triples of coordinates is given by the following 
\emph{sliding map} $F$, which can be applied to any two consecutive components of the 
triple as follows: 
\begin{equation}\label{e:slide}
F \left( \begin{pmatrix}
p\\
q
\end{pmatrix}_\delta,
\begin{pmatrix}
p_0\\
q_0
\end{pmatrix}_{\delta_0}\right)=
\left( \begin{pmatrix}
p_0\\
q_0
\end{pmatrix}_{\delta_0},
\begin{pmatrix}
p-\delta_0 \Delta_0 p_0\\
q-\delta_0 \Delta_0 q_0
\end{pmatrix}_{\delta}\right), \quad \text{where }\Delta_0=p_0q-q_0p.
\end{equation}

\begin{rem}\label{r:signs}
	Recall that the curves $\nu_i$ are oriented, and therefore so are the handles $h_{\nu_i}$. The sliding map $F$ describes the change of coordinates of the homology classes of the attaching curves as a result of a {\em handle addition} of oriented 2-handles (cf.~\cite[\S 5.1]{GS94}). On the other hand, the 4-manifold resulting from attaching a 2-handle does not depend on the choice of an orientation on the 2-handle, therefore a triple as in (\ref{e:handles}) can be modified by changing the signs of a pair $(p_i,q_i)$ (but not $\delta_i$) without changing the resulting 4-manifold up to diffeomorphisms.
\end{rem}

Our strategy to prove that $\widehat X$ is diffeomorphic to $\CP^2$ will be as follows. 
We will exhibit a sequence of slides such that the coordinates $p_i$ and $q_i$ gradually get smaller, until we end up with a familiar Kirby diagram for $\CP^2$.

We now show that, for any pair $(k,m)$ as above, the map $F$ transforms the starting triple $(\nu_1,\nu_2,\nu_3)$ into
\[
\left( \begin{pmatrix}
0\\
1
\end{pmatrix}_1,
\begin{pmatrix}
1\\
0
\end{pmatrix}_{-1},
\begin{pmatrix}
1\\
0
\end{pmatrix}_1\right),
\]
In order to do that we need to determine the coordinates $(p_i,q_i)$ of~\eqref{e:handles} in terms of $k$ and $m$. These will be given by products of $2 \x 2$ matrices as in Proposition~\ref{p:slamdunk}. Since the substring $(2^{[m-1]},m+2)$ occurs repeatedly in $s_{k,m}$, it will be useful to find a general formula for $A_2(A_2^{m-1}A_{m+2})^l$ (recall Notation~\ref{e:2x2}). For this purpose, observe that an obvious induction gives
\[
A_2^{m-1} = 
\mat{m&1-m\\m-1&2-m}.
\]
Then, define
\[
C:=\begin{pmatrix}
x+1 & -1\\
x & -1
\end{pmatrix}\in \text{GL}_2(\Z [x]).
\]
It is easy to check that the matrix $A_2^{m-1}A_{m+2}$ is obtained by evaluating the entries of $C^2$ at $m$. 

Now for $l\in\Z$ let $M_l := A_2 C^l$ and define the $\Z$-indexed sequences of polynomials $(P_l)$, $(Q_l)$, 
$(S_l)$ and $(T_l)$ by setting  
\begin{equation} \label{e:A2Bk}
\mat{P_l& -S_l\\Q_l& -T_l} := M_l.
\end{equation} 
Since $C^2 = xC + I$ we have $M_{l+2} = xM_{l+1} + M_l$, therefore each sequence 
satisfies the recursive formula
\begin{equation}\label{e:recursion}
f_{l+2}=x \cdot f_{l+1}+f_l.
\end{equation}
Such sequences are completely determined by their values at two adjacent indices. Moreover, 
\begin{itemize}
	\item by setting $l=0$ in (\ref{e:A2Bk}), we immediately get $P_0=2$, $Q_0=S_0=1$ and $T_0=0$;
	\item by setting $l=1$ and computing $A_2 C$ we get $P_1=x+2$, 
	$Q_1=x+1$ and $S_1=T_1=1$. 
\end{itemize}
The following table shows a few terms of the four sequences:
\begin{center}
	\begin{tabular}{|c || c | c | c | c |} 
		\hline
		$l$ & $P_l$ & $Q_l$ & $S_l$ & $T_l$ \\ %[0.5ex]
		\hline\hline
		$-1$ & $-x+2$ & $1$ & $1-x$ & $1$\\ 
		\hline
		$0$ & $2$ & $1$ & $1$ & $0$\\
		\hline
		$1$ & $x+2$ & $x+1$ & $1$ & $1$\\
		\hline
		$2$ & $x^2+2x+2$ & $x^2+x+1$ & $x+1$ & $x$\\
		%\hline
		%3 & $x^3+2x^2+3x+2$ & $x^3+x^2+2x+1$ & & & $x^3+2x$ \\ %[1ex] 
		\hline
	\end{tabular}
\end{center}
The values in the table together with~\eqref{e:recursion} imply that $S_l=Q_{l-1}$, therefore 
\[
M_l = A_2 C^l=\begin{pmatrix}
P_l & -Q_{l-1}\\
Q_l & -T_l
\end{pmatrix} \quad \forall l \in \Z.
\]

\begin{lemma}\label{l:identities}
	The sequences $(P_l)$, $(Q_l)$ and $(T_l)$ satisfy the following identities:
	\[
	(1)\hspace{5pt} P_{l+1}-P_l=x \cdot Q_l;\quad
	(2)\hspace{5pt} Q_{l+1}-Q_l=x \cdot T_{l+1};\quad 
	(3)\hspace{5pt} Q_{l+1}+Q_l=P_{l+1};\quad  
	(4)\hspace{5pt} T_l+T_{l-1}=Q_l; 
	\]
	\[
	(5)\hspace{5pt} P_{l+1}Q_l-P_lQ_{l+1}=(-1)^{l+1}x;\quad 
	(6)\hspace{5pt} Q_{2l}Q_{2l-1}-P_{2l}T_{2l}=1;\quad 
	(7)\hspace{5pt} P_{2l}T_{2l-1}-Q_{2l-1}^2=1.
	\]
\end{lemma}

\begin{proof}
	Both sides of (1), (2), (3) and (4) are the terms of two sequences of polynomials satisfying  the recursive formula (\ref{e:recursion}), therefore it is enough to verify the identities for two distinct values of $l$, say $0$ and $1$.
	(5) We first claim that $(P_{l+1}Q_l-P_lQ_{l+1})$ is a geometric progression with common ratio $-1$: by (\ref{e:recursion}), we have
	\[
	P_{l+1}Q_l-P_lQ_{l+1}=(x P_l+P_{l-1})Q_l-P_l (xQ_l+Q_{l-1})=-(P_{l}Q_{l-1}-P_{l-1}Q_l),
	\]
	which proves the claim. Now it is enough to verify the identity for $l=0$.
	(6) The LHS can be written as $\det(M_{2l}) = \det(A_2 C^{2l})$, which is immediately seen to be $1$, since $\det(A_2)=\det(C^2)=1$.
	Finally, (7) follows from (6) by substituting $Q_{2l}=P_{2l}-Q_{2l-1}$ and $T_{2l}=Q_{2l}-T_{2l-1}$, which is allowed by (3) and (4). 
\end{proof}

We can now compute the coordinates $(p_i,q_i)$ of $\nu_1$, $\nu_2$ and $\nu_3$:
\begin{itemize}
	\item $\nu_3$ is given by the first column of $A_2^{m-1}$, which is $\begin{pmatrix} m \\ m-1\end{pmatrix}$;
	\item $\nu_2$ is given by the first column of $A_2 C^{2k+2}|_{x=m} = M_{2k+2}|_{x=m}$, which is $\begin{pmatrix} P_{2k+2}(m) \\ Q_{2k+2}(m)\end{pmatrix}$;
	\item $\nu_1$ is given by the first column of 
	$A_2 C^{2k+2}|_{x=m} A_1 A_2^{m-1} = M_{2k+2}|_{x=m} A_1 A_2^{m-1}$, which is  
	\[
	M_{2k+2}|_{x=m} A_1 \begin{pmatrix}m \\m-1\end{pmatrix} = 
	\begin{pmatrix} 
	P_{2k+2}(m) & -Q_{2k+1}(m)\\
	Q_{2k+2}(m) & -T_{2k+2}(m)
	\end{pmatrix} 
	\begin{pmatrix}1 \\m \end{pmatrix} = 
	\begin{pmatrix}
	P_{2k+1}(m)\\
	Q_{2k+1}(m)
	\end{pmatrix},
	\] 
	where the last equality holds by Identities (1) and (2) of Lemma \ref{l:identities}.
\end{itemize}
Therefore, if for any pair $(l,m)$ of integers we define
\[ \tau_{l,m}:=\left( \begin{pmatrix}
P_{l+1}(m)\\
Q_{l+1}(m)
\end{pmatrix}_{(-1)^l},
\begin{pmatrix}
P_{l+2}(m)\\
Q_{l+2}(m)
\end{pmatrix}_{(-1)^{l+1}},
\begin{pmatrix}
m\\
m-1
\end{pmatrix}_1\right), \]
the starting triple that arises from $B(s_{k,m})$ via the previous construction is $\tau_{2k,m}$.

\begin{lemma}\label{l:km-ind}
	The following hold:
	\begin{itemize}
		\item[(1)] $\tau_{l,m}$ can be transformed into $\tau_{l+1,m}$ by applying the sliding map $F$ to the first two components;
		\item[(2)] any pair of the form $\left(\begin{pmatrix}
		a+2\\
		a+1
		\end{pmatrix}_1,
		\begin{pmatrix}
		a\\
		a-1
		\end{pmatrix}_1\right)$ can be transformed into $\left(\begin{pmatrix}
		a\\
		a-1
		\end{pmatrix}_1,
		\begin{pmatrix}
		a-2\\
		a-3
		\end{pmatrix}_1\right)$ by applying $F$ and changing the signs in the second component;
		%as in Remark \ref{r:signs};
		\item[(3)] $F^3 \left( \begin{pmatrix}
		0\\
		1
		\end{pmatrix}_1,
		\begin{pmatrix}
		1\\
		0
		\end{pmatrix}_{-1}
		\right)=
		\left( \begin{pmatrix}
		2\\
		1
		\end{pmatrix}_{-1},
		\begin{pmatrix}
		3\\
		2
		\end{pmatrix}_1
		\right)$.
	\end{itemize}
\end{lemma}
\begin{proof}
	We immediately see from (\ref{e:slide}) that $\tau_{l,m}$ is transformed into
	
	\[ \left( \begin{pmatrix}
	P_{l+2}(m)\\
	Q_{l+2}(m)
	\end{pmatrix}_{(-1)^{l+1}},
	\begin{pmatrix}
	P_{l+1}(m)-\delta_0 \Delta_0 P_{l+2}(m)\\
	Q_{l+1}(m)-\delta_0 \Delta_0 Q_{l+2}(m)
	\end{pmatrix}_{(-1)^{l}},
	\begin{pmatrix}
	m\\
	m-1
	\end{pmatrix}_1\right), \]
	which clearly agrees with $\tau_{l+1,m}$ at the first and the third components and at the framing of the second one. Therefore, we are left with verifying that
	
	\[
	P_{l+1}(m)-\delta_0 \Delta_0 P_{l+2}(m)=P_{l+3}(m) \quad \text{and} \quad Q_{l+1}(m)-\delta_0 \Delta_0 Q_{l+2}(m)=Q_{l+3}(m).
	\]
	
	By (\ref{e:recursion}), both these equalities follow from $\delta_0 \Delta_0=-m$: this is true because
	
	\[
	\delta_0 \Delta_0=(-1)^{l+1} (P_{l+2}(m)Q_{l+1}(m)-P_{l+1}(m)Q_{l+2}(m))=(-1)^{l+1}(-1)^l m=-m
	\]
	where the second equality holds by Lemma \ref{l:identities}(5). This proves (1).
	Finally, (2) and (3) follow from a straightforward computation; in particular, in order to prove (3), it is useful to observe that the quantity $\delta_0 \Delta_0$ stays unchanged at each step, since both $\delta_0$ and $\Delta_0$ change sign.
\end{proof}

%\begin{proof}[Proof of Theorem~\ref{t:main1}]
Now, in  order to prove Theorem 1(1), we must show that the triple $\tau_{2k,m}$ corresponds to a Kirby diagram for $\CP^2$. By Lemma \ref{l:km-ind}(1), it is enough to prove this for
\[
\tau_{-1,m}=\left( \begin{pmatrix}
2\\
1
\end{pmatrix}_{-1},
\begin{pmatrix}
m+2\\
m+1
\end{pmatrix}_1,
\begin{pmatrix}
m\\
m-1
\end{pmatrix}_1\right).
\]
We can apply Lemma~\ref{l:km-ind}(2) several times to the last two components. Observe that all coordinates decrease by $2$ at each step and recall that $m$ is odd. After $\frac{m-1}{2}$ applications of Lemma~\ref{l:km-ind}(2) we get 
\[
\left( \begin{pmatrix}2\\1\end{pmatrix}_{-1},
\begin{pmatrix}3\\2\end{pmatrix}_1,
\begin{pmatrix}1\\0\end{pmatrix}_1\right),
\]
and finally, applying Lemma~\ref{l:km-ind}(3) to the first two components,
\[
\left( \begin{pmatrix}0\\1\end{pmatrix}_1,
\begin{pmatrix}1\\0\end{pmatrix}_{-1},
\begin{pmatrix}
1\\0\end{pmatrix}_1\right).
\]

The last step in the proof of Theorem~\ref{t:main1}(1) is the following:

\begin{lemma} \label{l:cp2}
	The triple $\left( \begin{pmatrix}
	0\\
	1
	\end{pmatrix}_1,
	\begin{pmatrix}
	1\\
	0
	\end{pmatrix}_{-1},
	\begin{pmatrix}
	1\\
	0
	\end{pmatrix}_1\right)$ corresponds to a Kirby diagram for $\CP^2$.
\end{lemma}

\begin{proof}
	We have three knots in $T^2 \x [0,1] \subset S^1 \x S^2$, which can be glued to $V_1$ along $T^2 \x \{0\}$ to form a new solid torus, which we regard as the exterior of an unknot $\widehat{K}$ in $S^3$ (as in the proof of Proposition~\ref{p:slamdunk}). Consequently, we can regard $S^1 \x S^2$ as the result of a Dehn surgery along $\widehat{K}$ with framing $0$. Now, the attaching curves $\nu_1$, $\nu_2$ and $\nu_3$ of the $2-$handles are contained in three nested tori, each of which bounds a regular neighborhood of $\widehat{K}$. More precisely, $\nu_1$ is a parallel copy of $m_1$, hence a canonical longitude of $\widehat{K}$, while $\nu_2$ and $\nu_3$ are two parallel copies of $\ell_1$, hence two unlinked meridians of both $\widehat{K}$ and $\nu_1$. The left-most picture of Figure~\ref{fig:diag} illustrates the resulting handlebody decomposition. 
	\begin{figure}[ht]
		\labellist
		\pinlabel $\cup 3h\cup 4h$ at 60 -10
		\pinlabel $\cup 3h\cup 4h$ at 220 10
		\pinlabel $\cup 3h\cup 4h$ at 343 25
		\pinlabel $\cup 4h$ at 422 25
		\pinlabel $1$ at 20 82
		\pinlabel $-1$ at 88 65
		\pinlabel $1$ at 120 75
		\pinlabel $1$ at 215 80
		%\pinlabel $1$ at 185 55
		\pinlabel $0$ at 215 55
		\pinlabel $1$ at 240 57
		\pinlabel $0$ at 324 67
		\pinlabel $1$ at 361 67
		\pinlabel $1$ at 420 67
		\endlabellist
		\centering
		\includegraphics{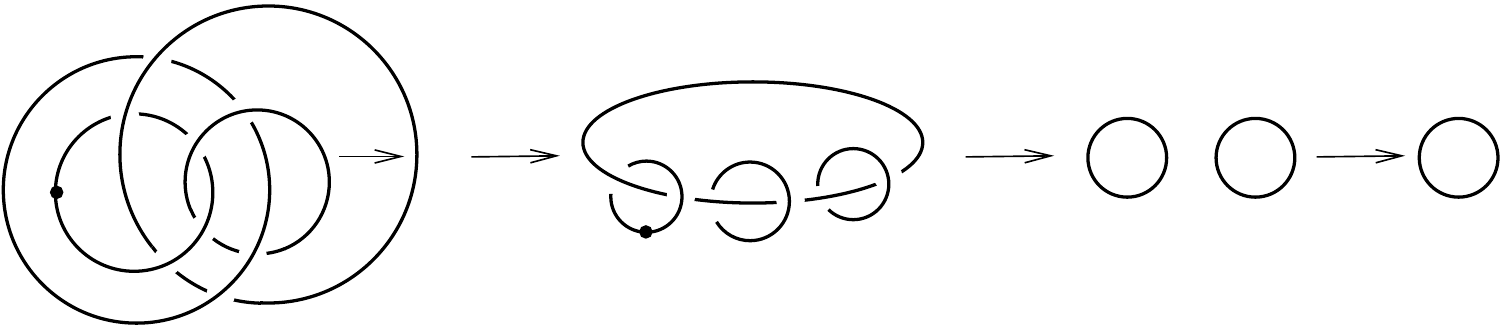}
		\caption{}
		\label{fig:diag}
	\end{figure}
	Performing the handle slide indicated by the horizontal arrow yields the second picture of Figure~\ref{fig:diag}, canceling the obvious $1$-$2$-handle pair yields the third picture, and canceling the $0$-framed unknot with the $3$-handle gives the well known Kirby diagram for $\CP^2$.
\end{proof}

By Lemmas~\ref{l:km-ind} and~\ref{l:cp2}, the 4-manifold $\widehat X$ is diffeomorphic to $\CP^2$. This proves the existence of the smooth embeddings, i.e.~Part (1) of Theorem~\ref{t:main1}.

Part (2) of Theorem~\ref{t:main1} follows from~\cite[Theorem~1]{Ow20} if $m=1$, so in the following we assume $m \ge 3$. By the results of Evans and Smith~\cite{ES18} recalled in Section~\ref{s:intro}, to show that $B(s_{k,m})$ does not symplectically embed in $\CP^2$ it suffices to write the lens space $L(s_{k,m})=\del B(s_{k,m})$ as $L(p^2,pq-1)$ and show that $p$ does not divide $q^2+9$. By Proposition~\ref{p:slamdunk} we can find such $p$ and $q$ by computing the first column of 
%A_2 B_m^{k+1} A_2^2 B_m^{k+1}$: we have
$M_{2k+2} A_2 M_{2k+2}|_{x=m}$: we have 
\[
M_{2k+2} A_2 M_{2k+2}\mat{1\\0}=
\begin{pmatrix}
P_{2k+2} & -Q_{2k+1}\\
Q_{2k+2} & -T_{2k+1} 
\end{pmatrix}
\begin{pmatrix}
2 & -1\\
1 & 0 
\end{pmatrix}
\begin{pmatrix}
P_{2k+2}\\
Q_{2k+2} 
\end{pmatrix}=
\]
\[
=\begin{pmatrix}
P_{2k+2} & -Q_{2k+1}\\
Q_{2k+2} & -T_{2k+2} 
\end{pmatrix}
\begin{pmatrix}
2P_{2k+2}-Q_{2k+2}\\
P_{2k+2} 
\end{pmatrix}\stackrel{(3)}{=}
\begin{pmatrix}
P_{2k+2} & -Q_{2k+1}\\
Q_{2k+2} & -T_{2k+2} 
\end{pmatrix}
\begin{pmatrix}
P_{2k+2}+Q_{2k+1}\\
P_{2k+2} 
\end{pmatrix}=
\]
\[
=\begin{pmatrix}
P_{2k+2}^2\\
P_{2k+2}(Q_{2k+2}-T_{2k+2})+Q_{2k+1}Q_{2k+2} 
\end{pmatrix}\stackrel{(6)}{=}
\begin{pmatrix}
P_{2k+2}^2\\
P_{2k+2}Q_{2k+2}+1 
\end{pmatrix}.
\]
The numbers above the equality symbols denote which identities from Lemma~\ref{l:identities} have been used. We can now obtain the first column of $M_{2k+2} A_2 M_{2k+2}|_{x=m}$ by evaluating the above polynomials at $m$. We obtain $p=P_{2k+2}(m)$ and $q=P_{2k+2}(m)-Q_{2k+2}(m)\stackrel{(3)}{=} Q_{2k+1}(m)$. By Lemma \ref{l:identities}(7),
\[
q^2+9=Q_{2k+1}(m)^2+9=P_{2k+2}(m) T_{2k+1}(m)+8,
\]
which is a multiple of $P_{2k+2}(m)$ if and only if $P_{2k+2}(m)\ |\ 8$. However, we can easily observe that, for each $l \ge 1$, $P_l$ is a monic polynomial of degree $l$ with positive coefficients, so that $P_{2k+2}(m) \ge m^{2k+2} \ge m^2 \ge 9$, and in particular $P_{2k+2}(m) \nmid 8$. This concludes the proof of Theorem~\ref{t:main1}.

\bibliographystyle{abbrv}
\bibliography{biblio}

\begin{thebibliography}{10}

\bibitem{ES18}
J.~Evans and I.~Smith.
\newblock Markov numbers and {L}agrangian cell complexes in the complex
  projective plane.
\newblock {\em Geometry \& Topology}, 22(2):1143--1180, 2018.

\bibitem{GS94}
R.~E. Gompf and A.~I. Stipsicz.
\newblock {\em 4-manifolds and {K}irby calculus}.
\newblock Number~20 in Graduate Studies in Mathematics. American Mathematical
  Soc., 1999.

\bibitem{HP10}
P.~Hacking and Y.~Prokhorov.
\newblock Smoothable del {P}ezzo surfaces with quotient singularities.
\newblock {\em Compositio Mathematica}, 146(1):169--192, 2010.

\bibitem{Kh13}
T.~Khodorovskiy.
\newblock Symplectic rational blow-up.
\newblock arXiv:1303.2581, 2013.

\bibitem{KM94}
R.~Kirby and P.~Melvin.
\newblock Dedekind sums, $\mu$-invariants and the signature cocycle.
\newblock {\em Mathematische Annalen}, 299(2):231--268, 1994.

\bibitem{Ko08}
J.~Koll\'ar.
\newblock Is there a topological {B}ogomolov-{M}iyaoka-{Y}au inequality?
\newblock {\em Pure and Applied Mathematics Quarterly}, 4(2):203--236, 2008.

\bibitem{Li08}
P.~Lisca.
\newblock On symplectic fillings of lens spaces.
\newblock {\em Transactions of the American Mathematical Society},
  360(2):765--799, 2008.

\bibitem{NP10}
A.~N{\'e}methi and P.~Popescu-Pampu.
\newblock On the {M}ilnor fibres of cyclic quotient singularities.
\newblock {\em Proceedings of the London Mathematical Society},
  101(2):554--588, 2010.

\bibitem{Ow20}
B.~Owens.
\newblock Smooth, nonsymplectic embeddings of rational balls in the complex
  projective plane.
\newblock {\em The Quarterly Journal of Mathematics}, 71(3):997--1007, 2020.

\bibitem{Ri74}
O.~Riemenschneider.
\newblock Deformationen von {Q}uotientensingularit{\"a}ten (nach zyklischen
  gruppen).
\newblock {\em Mathematische Annalen}, 209(3):211--248, 1974.

\bibitem{TY12}
M.~Tange and Y.~Yamada.
\newblock Four-dimensional manifolds constructed by lens space surgeries along
  torus knots.
\newblock {\em Journal of Knot Theory and Its Ramifications}, 21(11):1250111,
  2012.

\end{thebibliography}
\end{document}